\documentclass[12pt]{amsart}

\usepackage{amsmath, amssymb}
\usepackage{enumitem}
\newtheorem{theorem}{Theorem}[section]
\newtheorem{lemma}[theorem]{Lemma}

\newtheorem{cor}[theorem]{Corollary}

\theoremstyle{definition}

\theoremstyle{remark}
\newtheorem{remark}[theorem]{Remark}

\numberwithin{equation}{section}

\newcommand{\rd}{{\mathbb R^d}}
\newcommand{\zd}{{\mathbb Z^d}}

\newcommand{\N}{{\mathbb N}}
\newcommand{\nat}{{\mathbb N}}

\allowdisplaybreaks

\begin{document}
\sloppy
\title[Graph distances of continuum long-range percolation]{Graph distances of continuum long-range percolation} 

\author{Ercan S\"onmez}
\address{Ercan S\"onmez, Department of Statistics, University of Klagenfurt, Universit\"atsstraße 65--67, 9020 Klagenfurt, Austria}
\email{ercan.soenmez\@@{}aau.at} 


\begin{abstract}
We consider a version of continuum long-range percolation on finite boxes of $\rd$ in which the vertex set is given by the points of a Poisson point process and each pair of two vertices at distance $r$ is connected with probability proportional to $r^{-s}$ for a certain constant $s$. We explore the graph-theoretical distance in this model. The aim of this paper is to show that this random graph model undergoes phase transitions at values $s=d$ and $s=2d$ in analogy to classical long-range percolation on $\zd$, by using techniques which are based on an analysis of the underlying Poisson point process.
\end{abstract}

\maketitle

\baselineskip=18pt

\section{Introduction}

\textit{Small-world phenomenon} is a socio-psychological notion originated from a series of social experiments conducted by Milgram \cite{Mil}. These experiments suggest that human society is a network of 'small-world' type in the sense that arbitrary groups or persons are connected by very short paths. The problem of small degrees of separation arises in various networks \cite{Watts1, Watts3} and has led to many interesting yet open challenges in probability theory.

A classical approach modeling (social) networks is to use random graph models \cite{Watts2}. In such a model the node set may correspond to a geographical location and the edges correspond to links. An accepted model taking into account the property of decreasing distances is long-range percolation \cite{Schul, NewSchul, AizNew} in which the node set is assumed to be the integer lattice $\mathbb{Z}^d$ and the probability of an edge between two arbitrary nodes asymptotically has a polynomial decay in their distances, i.e. for every $u$ and $v$ in $\zd$ there is an edge connecting $u$ and $v$ with some probability $p(u,v)$ only depending on the Euclidean distance of $u-v$ and the origin such that
$$p(u,v) \sim \beta \| u-v\|^{-s}$$
for certain constants $\beta$ and $s$. One way of measuring distances is to use the graph-theoretical distance defined as the minimal number of edges used in a path connecting an arbitrary pair of nodes.

Extending the work of \cite{BenBer} the authors in \cite{CGS} studied the distance scaling of long-range percolation on finite boxes under the assumption that nearest neighbor edges are present almost surely. In particular, their results have revealed four distinct regimes of behavior, namely $s=d$, $d < s < 2d$, $s=2d$ and $s >2d$.

In this contribution we focus on a continuous analogue of long-range percolation on $\mathbb{Z}^d$, the main and crucial difference being that the node set is randomly scattered over $\rd$ as the realization $\mathcal{P}$ of a homogeneous Poisson point process. One on the first papers which rigorously studied such a model is \cite{Penrose}. See \cite{Last, Soen} and references therein for recent results on the \textit{random connection model}. Given a realization of $\mathcal{P}$ in the random connection model an edge between two nodes $x,y \in \mathcal{P}$ is present with some probability $g(x-y)$ only depending on the Euclidean distance $\| x-y \|$. Under some integrability conditions on the function $g$, see \cite[p. 536]{Penrose}, the random connection model has the specific feature that the occurrence of percolation, i.e. existence of an infinite cluster, only depends on the intensity constant of $\mathcal{P}$. Thus, percolation may occur for edge probabilities under which there is no percolation in classical long-range percolation on $\zd$. Following and motivated by discrete long-range percolation in this paper we mainly consider the case
\begin{align} \label{revg}
g(x) = 1 - \exp ( -\beta \|x\|^{-s} ), \quad x \in \rd,
\end{align}
for certain $\beta,s \in (0, \infty)$. Our main target is the behavior of the chemical distance of a random connection model in finite boxes. It is of interest to study such distances and to compare them to the results obtained in \cite[Theorem 3.1]{CGS}. This is the main goal of this paper.

Let us finally remark that a slightly similar question regarding chemical distances of long-range percolation on $\rd$ has been recently addressed in \cite{Biskup} using different methods. In contrary to the present article the focus there is on the case $s \in (d, 2d)$. For related results we also refer to \cite{DW}, which is based on \cite{DHH} and introduces an inhomogeneous version of the random connection model. In the model they introduce the vertices are equipped with random iid weights and each pair of vertices is connected with probability increasing in the vertex weights, see Section 2 below for details. In \cite{DW} the authors also address questions regarding chemical distances in several regimes. In some regimes their results are of the same type as some results in this paper.
Moreover, we will borrow ideas from \cite{CGS} and adapt them to our setting. This results in the fact that our calculations are different and crucially rely on the structure of the vertex set which is characterized by the underlying Poisson point process.

We end this section with an outline of the remainder of this article. In the following section we rigorously introduce the model under investigation and formulate the main results. The proceeding sections are divided into the proofs of the individual cases of our main Theorem.

\section{Model and main results}

We construct a random graph in the following way. The vertex set is given by points of a Poisson point process and the edge set is given by connecting the nodes independently. Let $\mathcal{P}$ be a homogeneous Poisson point process with intensity $\rho >0$, that is $\mathcal{P}$ satisfies the following \cite{Kall}:
\begin{itemize}
	\item[(i)] For every set $B \in \mathcal{B} (\rd)$ the random variable $\mathcal{P} (B)$ has a Poisson distribution with parameter $\rho |B|$ with $|B|$ denoting the Lebesgue measure of the Borel set $B$ and $\mathcal{B} (\rd)$ denoting the Borel $\sigma$-algebra of $\rd$.
	\item[(ii)] For every $n \in \nat$ and disjoint sets $B_1, \ldots, B_n \in \mathcal{B} (\rd)$ the random variables\newline $\mathcal{P} (B_1), \ldots, \mathcal{P} (B_n)$ are independent.
\end{itemize}
Given a realization of $\mathcal{P}$ we write $\mathcal{P} = (X_n)_{n \in \mathbb{N}}$, see \cite[Corollary 6.5]{LP}. Thus, the vertex set will be given by $\mathcal{P}$.

From now on, denote by $\| \cdot \| = \| \cdot \|_1$ the $l_1$-norm in $\rd$. Let $g \colon \rd \to [0,1]$ be given by \eqref{revg}. Consider a set of Bernoulli random variables $(E_{x,y}: \{x,y\} \in \mathcal{A}, x \neq y)$ on a probability space $(\Omega, \mathcal{F}, P)$, where the indexing set $\mathcal{A}$ consists of all unordered pairs $\{x,y\}$ of elements of $\rd$. We use the Kolmogorov consistency theorem and choose $(E_{x,y}: \{x,y\} \in \mathcal{A}, x \neq y)$ such that $P (E_{x,y} = 1) = g(x-y)$ for all $\{x,y\} \in \mathcal{A}$ with $x \neq y$, independently. Thus we obtain a random graph with vertex set given by the points  $(X_n)_{n \in \mathbb{N}}$ of $\mathcal{P}$ and by including an edge $(X_i, X_j)$ if and only if $E_{X_i, X_j} = 1$. By abuse of notation we denote the joint probability measure of the point process $\mathcal{P}$ with intensity $\rho >0$ and edge occupation by $P$. We denote by $\mathcal{C}(0)$ the connected component, referred to as the cluster, containing the origin (if it belongs to $\mathcal{P}$). Note that $\mathcal{C}(0)$ is almost surely connected by definition. In \cite[Theorem 1]{Penrose} it is shown that under the condition
\begin{align} \label{revg2}
0 < \int_\rd g(x) dx < \infty
\end{align}
there exists a critical intensity of the Poisson point process denoted $\rho_c$ with $\rho_c \in (0,\infty)$ such that
$$P \big( |\mathcal{C}(0)| = \infty \big) >0$$
for all $\rho > \rho_c$. Moreover, there exists a unique infinite connected component with probability one. Note that \eqref{revg2} is fulfilled for $s>d$.

Given a graph $G=(V,E)$ the graph distance on $G$ between two nodes $x,y \in V$ is defined as $D(x,y)$ corresponding to the number of edges in $E$ on a shortest path between $x$ and $y$, with the convention that $D(x,x) = 0$ and $D(x,y) = \infty$ if $x$ and $y$ are not in the same connected component. The diameter of $G$ denoted by $D_G$ is defined as
$$ D_G = \max_{x,y \in V} D(x,y),$$
i.e. the maximal graph distance between two nodes in $G$. 

Let $N$ be an integer and consider a partition of the box $[0,N]^d$ into $N^d$ subcubes $I_1, \ldots, I_{N^d}$ each of unit side length. Let $x_i \in I_i$ for $i=1, \ldots, N^d$ and denote by $\mathcal{C}_N$ the largest connected component in $[0,N]^d$ (with a deterministic rule if there is more than one largest connected component). Let $D(N)$ be the random diameter $D_{\mathcal{C}_N}$ of $\mathcal{C}_N$. In order to measure events involving $D(0,x)$ for $x \in \rd$ we need to ensure that the origin and the point $x$ are elements of the vertex set. Therefor, we also consider the probability measure $P_{0,x}( \cdot)=P( \cdot| 0,x \in \mathcal{P})$, which is known as the twofold Palm measure. Observe that $ P( \cdot| 0,x \in \mathcal{C}_\infty) = P_{0,x}( \cdot| 0,x \in \mathcal{C}_\infty)$, where $\mathcal{C}_\infty$ denotes the (unique) infinite connected component if it exists. We now state our main results as follows.

\begin{theorem} \label{MR}
	There exist positive and finite constants $C_1, C_2, C_3$ and $0 < \eta < 1$ depending on $\beta$ and $d$ such that the following hold:
	\begin{itemize}
		\item [1.] If $s > 2d$ then
		$$ \lim_{\|x\| \to \infty} P \Big( D(0,x) \geq C_1 \|x\| \Big| 0,x \in \mathcal{P} \Big) = 1.$$
		\item[2.] If $s=2d$ and $d=1$ then
		$$ \lim_{N \to \infty} P\Big( D(N) \geq N^{\eta}   \Big) = 1.$$
		\item[3.] Let $d<s<2d$.
		\begin{itemize}
			\item [(i)] For all $\varepsilon>0$ it holds
			$$ \lim_{\|x\| \to \infty} P \Big( \log D(0,x) \geq (1- \varepsilon) \log \log \|x\| \Big| 0,x \in \mathcal{P} \Big) = 1.$$
			\item [(ii)] Let $\rho > \rho_c$. For all $\varepsilon>0$ it holds
			$$ \lim_{\|x\| \to \infty} P \Big( \log D(0,x) \leq (1+ \varepsilon) \frac{\log 2}{\log (2d/s)} \log \log \|x\| \Big| 0,x \in \mathcal{C}_\infty \Big) = 1,$$
			where $\mathcal{C}_\infty$ denotes the (unique) infinite connected component (which almost surely exists).
		\end{itemize}
		
		\item[4.] Let $s=d$.
		\begin{itemize}
			\item [(i)] It holds
			$$ \lim_{N \to \infty} P\Big( D(N) \geq C_2 \frac{\log N}{\log \log N}  \Big) = 1.$$
			\item [(ii)] Let $\rho > \rho_c$ as above. It holds
			$$ \lim_{\|x\| \to \infty} P \Big(  D(0,x) \leq C_3 \frac{\log \|x\|}{\log \log \|x\|} \Big| 0,x \in \mathcal{C}_\infty \Big) = 1,$$
			where $\mathcal{C}_\infty$ denotes an infinite connected component (which almost surely exists).
		\end{itemize}
	\end{itemize}
\end{theorem}

Here are some comments regarding our main result. In Theorem \ref{MR} we are able to provide an analogy that in our model the random graph behaves similarly to classical discrete long-range percolation as in \cite[Theorem 3.1]{CGS} and \cite[Theorem 1]{Bn1}. 
Part 3 (ii) of Theorem \ref{MR} is due to the proof of the upper bound in \cite[Theorem 3.6 (b1)]{DW}. Let us explain this in some detail. As mentioned in the introduction \cite{DW} considers an inhomogeneous version of our model, in which two vertex points $x,y \in \mathcal{P}$ are connected with probability
$$1- \exp(- \beta W_x W_y \| x-y\|^{-s}),$$
where the weights $W_x, x \in \mathcal{P}$, are iid random variables with values in $[0, \infty)$, see \cite[Section 2]{DW} for more details. There, the proof of the bound in part 3 (ii) relies on \cite[Theorem 3.4 and Lemma 7.7]{DW}. In both of these proofs the authors (1) bound the connection probability in the inhomogeneous model by the connection probability of the (homogeneous) random connection model and (2) make use of Chernoff bounds for Poisson random variables in order to estimate the number of vertices in boxes. They prove this specific result for the homogeneous model and, hence, the proof of part 3 (ii) follows one-to-one from the proof of the upper bound on the graph distance in \cite[Theorem 3.6 (b2)]{DW}. The main contribution in the following are the proofs of all the remaining parts. It will be clear that our methods rely on \cite{CGS, DWdisc, DW}.

Let us finally remark that we do not investigate the diameter in the case $s < d$. For classical long-range percolation it was shown in \cite{BenKes} that for $s <d$ the diameter is $\lceil \frac{d}{d-s} \rceil$ with high probability. We believe that this might be true in our model as well, which is an interesting problem to investigate, but beyond the scope of this paper.

From now on throughout this entire manuscript we denote by $c$ a universal constant which might be different in each occurrence. Moreover, for all events $A_N$ depending on the integer $N$ we say that $A_N$ occurs with high probability if $P(A_N) \to 1$ as $N \to \infty$. Furthermore, we will use the notation $x \leftrightarrow y$ if there is an edge connecting the vertices $x$ and $y$. Moreover, we will use the notation $x \nleftrightarrow y$ if such an edge is absent. We proceed by proving part 1 of Theorem \ref{MR} in the following section. 

\section{The case $s > 2d$}
\noindent \textit{Proof of Theorem \ref{MR}, part 1.} We note that the claim to be proven is the continuous analogue of \cite[Theorem 1]{Bn1} and our proof is an adaption to our model. Moreover, such an adaption has also been carried out in \cite[Theorem 8 (b2)]{DWdisc} and \cite[Theorem 3.6 (b2)]{DW}. The idea is to use a renormalization technique which we outline in detail now.

We consider a sequence $(a_n)_{n \in \mathbb{N}_0}$ with values in $\mathbb{N}$ and another sequence $(m_n)_{n \in \mathbb{N}_0}$ with $m_0 = a_0$ and
$$m_n = \prod_{i=0}^{n} a_i = a_n m_{n-1}, \quad n \in \mathbb{N}.$$
Then we define so-called $n$-stage boxes, $n \in \mathbb{N}_0$, as
$$B_{m_n} (x) = x + [0, m_{n}-1]^d, \quad x \in \rd.$$
Based on this definition for $n \in \mathbb{N}$ the disjoint $(n-1)$-stage boxes
$$B_{m_{n-1}} (x + y m_{n-1}) = x + y m_{n-1} [0, m_{n}-1]^d,$$
with $y \in [0, a_{n}-1]^d \cap \mathbb{Z}^d$, are called the children of the $n$-stage box $B_{m_n} (x)$, $x \in \rd$. Note that the total number of such children is $a_n^d$. Now we recall the definition of \textit{good} $n$-stage boxes \cite[Definition 12]{DWdisc}. Fix $n \in \mathbb{N}_0$ and $x \in \rd$.
\begin{itemize}
	\item [(i)] We say that a 0-stage box $B_{m_0} (x)$ is good if there is no edge in $B_{m_0} (x)$ with length larger than $\frac{m_0}{100}$.
	\item [(ii)] We say that an $n$-stage box $B_{m_n} (x)$, $n \in \mathbb{N}$, is good if for all $j \in \{-1,0,1\}^d$
	\begin{itemize}
		\item [1.] there is no edge in $B_{m_n} (x + j \frac{m_{n-1}}{2})$ with length larger than $\frac{m_{n-1}}{100}$ and
		\item [2.] there are at most $3^d$ children of $B_{m_n} (x + j \frac{m_{n-1}}{2})$ that are not good.
	\end{itemize}
\end{itemize}
Having recalled the definition of good boxes the proof now relies on showing a linear lower bound on the graph distance within good boxes and that centered boxes of sufficiently large side lengths are good. More specifically, we will invoke the following result, \cite[Lemma 15]{DWdisc} and \cite[Proposition 3]{Bn1}, which also has been used in the inhomogeneous version of our continuum model \cite[Theorem 3.6 (b2)]{DW}.

\begin{lemma} \label{mrev1}
	Assume that $a_n = n^2$ for $n \in \N$. Suppose there exists $n_0 \in \N$ such that for all $n \geq n_0$ the following hold:
	\begin{itemize}
		\item [1.] for all $j \in \{-1,0,1\}^d$ the $n$-stage box $B_{m_n} ( j \frac{m_{n}}{2})$ is good
		\item [2.] for all $l >n$ the $l$-stage boxes $\hat{B}_{m_l}$ centered at $B_{m_n} (0)$ are good.
	\end{itemize}
	Then there is a constant $C \in (0, \infty)$ such that $D(x,y) \geq C\|x-y\|$ if $x,y \in B_{m_n} (0)$ with $\|x-y\| > \frac{m_n}{8}$.
\end{lemma}

In order to apply Lemma \ref{mrev1} we prove the following assertion.

\begin{lemma} \label{mrev2}
	Let $r,t \in (0, \infty)$. Suppose that $\Lambda_r$ is a box of side length $r$. Then
	$$ P \Big( \text{there is an edge in } \Lambda_r \text{ with length at least } t \Big) \leq cr^d t^{d-s}$$
	for some constant $c$.	
\end{lemma}

\begin{proof}
	Let 
	$$M(t) = \sum_{x,y \in \mathcal{P} \cap \Lambda_r} \mathbf{1}_{\{ \|x-y\| \geq t \}} \mathbf{1}_{\{ x \leftrightarrow y \}}$$
	be the total number of edges in $\Lambda_r$ of length at least $t$. Recall that the Mecke formula, see \cite[Theorem 4.4]{LP}, allows to rewrite the expectation of a sum over tuples of points in a Poisson point process as an integral over tuples of points in $\rd$. Due to this, we obtain
	\begin{align*}
	\mathbb{E} [M(t)] \leq c r^d \int_{ \|z\| \geq t} g(z) dz \leq cr^d \int_{ \|z\| \geq t} \| z\|^{-s} dz = cr^d t^{d-s},
	\end{align*}
	where we used $1-e^{-x} \leq x$ in the last inequality and a change to polar coordinates. Thus, the Markov inequality yields
	$$ P \Big( M(t) \geq 1 \Big) \leq 	\mathbb{E} [M(t)] \leq c r^d t^{d-s}.$$
\end{proof}
Next we prove the analogue of \cite[Lemma 1]{Bn1} and \cite[Lemma 14]{DWdisc} in our model, which will enable us to conclude the proof.

\begin{lemma} \label{mrev3}
	Suppose $a_n = n^2$ for $n \in \N$. For $a_0$ sufficiently large it holds
	$$\sum_{n \in \N_0} P \Big( B_{m_n} (0) \text{ is not good} \Big) < \infty.$$	
\end{lemma}

\begin{proof}
	Define $p_n = P ( B_{m_n} (0) \text{ is not good} )$, $n \in \N_0$. We prove that for all $a_0 = m_0$ sufficiently large and $n \in \N_0$
	\begin{align} \label{mrev3p1}
	p_n \leq c_d (n+1)^{-4d} e^{-2n},
	\end{align}
	where $c_d \in (0, \infty)$ is independent of $n$. This is done by induction.
	
	For $n=0$ from Lemma \ref{mrev2} we obtain
	\begin{align} \label{mrev3p2}
	\begin{split}
	p_0 &= P \Big( \text{there is an edge in } B_{m_0} (0) \text{ with length larger than } \frac{m_0}{100} \Big) \\
	& \leq c m_0^d \Big( \frac{m_0}{100} \Big)^{d-s} < 3^{-d} 2^{-4d-1}e^{-2}
	\end{split}
	\end{align}
	once $m_0$ is sufficiently large. Since $a_1 =1$, $B_{m_1} (0)$ has only one child, yielding
	$$p_1 \leq 3^d p_0 \leq c3^d 2^{-8d-1}e^{-4}$$
	by using \eqref{mrev3p2} and $m_0$ sufficiently large. Thus, \eqref{mrev3p1} is true for all $n\in \{0,1 \}$. Now for the induction step assume that \eqref{mrev3p1} is true for all $0 \leq k \leq n-1$ with some $n \geq 2$. By definition the $n$-stage box $B_{m_n} (0)$ is not good if at least one of the $3^d$ translations $B_{m_n} (j \frac{m_{n-1}}{2})$, $j \in \{-1,0,1\}^d$, fails to have property 1. or 2. in the definition of good boxes from above. Set $\gamma = s-2d >0$. By Lemma \ref{mrev2} we get
	\begin{align*}
	p_n & \leq 3^d \Big( c a_n^{s-d} m_n^{-\gamma} + \\
	& \quad P \big( \text{there are at least $3^d+1$ children of $B_{m_n} (0)$ that are not good} \big) \Big)
	\end{align*}
	for $m_0$ sufficiently large. Note that the event in the latter probability implies that there are at least two children $B_{m_{n-1}} (y)$ and $B_{m_{n-1}} (z)$ of $B_{m_{n}} (0)$ that are not good and separated by distance at least $2m_{n-1}$. Therefor, since $m_i = a_0 (i!)^2$ by definition, $i \geq 0$, the events $\{B_{m_{n-1}} (y) \text{ is not good} \}$ and $\{B_{m_{n-1}} (z) \text{ is not good} \}$ are independent. Using this, we further obtain
	\begin{align*}
	p_n & \leq 3^d \Big( c a_n^{s-d} m_n^{-\gamma} + \binom{a_n^d}{2} p_{n-1}^2 \Big) \leq 3^d \Big( c a_n^{s-d} m_n^{-\gamma} + {a_n^{2d}} p_{n-1}^2 \Big) \\
	& =  3^d \Big( c n^{2(\gamma +d)} (m_0 (n!)^2)^{-\gamma} + n^{4d} p_{n-1}^2 \Big) = c3^d m_0^{-\gamma} n^{2(\gamma +d)} (n!)^{-2\gamma} + 3^d n^{4d} p_{n-1}^2.
	\end{align*}
	Thus, there exists $n_0 \in \N$ such that for all $n \geq n_0$ and $m_0$ sufficiently large
	\begin{align} \label{mrev3p4}
	p_n \leq 3^{-d} 2^{-4d-2} e^{-2} (n+1)^{-4d} e^{-2n}  + 3^d n^{4d} p_{n-1}^2.
	\end{align}
	Moreover, one can choose $m_0$ large enough such that \eqref{mrev3p4} also holds for all $2 \leq n \leq n_0$. Finally using this and the induction assumption we obtain
	\begin{align*}
	p_n & \leq  3^{-d} 2^{-4d-2} e^{-2} (n+1)^{-4d} e^{-2n}  + c_d^2 n^{-4d} e^{-4n+4} \\
	& \leq c (n+1)^{-4d} e^{-2n} + c n^{-4d} e^{-4n+4} \leq   c (n+1)^{-4d} e^{-2n} \Big( 1+ \big( \frac{n+1}{n} \big)^{4d} e^{-2n} \Big) \\
	& \leq c (n+1)^{-4d} e^{-2n}
	\end{align*}
	with $c$ a constant only depending on $d$. This proves \eqref{mrev3p1} and concludes the proof.
\end{proof}
Now we are in position to complete the proof of part 1 of our main Theorem. Lemma \ref{mrev3} implies that there is $n_0 \in \N_0$ such that for all $l \geq n_0$ the $l$-stage boxes $\hat{B}_{m_l}$ are good almost surely. Eventually, with probability one, by Lemma \ref{mrev1}, as the graph distance between good boxes is at least linear, for sufficiently large $n$ and $x \in \mathcal{P}$ with $\| x\| > \frac{m_n}{8}$ we have $D(0,x) \geq c \| x\|$ for some constant $c \in (0, \infty)$. \hfill $\Box$

\section{ The case $s=2d$}

The proof of the lower bound for the one-dimensional case $d=1$, $s=2$ uses the notion of an isolated point and the notion of a cut interval. 
Let $J \subset [0,N]$ be an interval. We say that a vertex $x \in \mathcal{P} \cap J$ is isolated with respect to $J$ if there is no edge connecting $x$ to another point $y \in \mathcal{P} \cap J$ and denote by $M(J)$ the number of such vertices. Before proving the lower bound we first show that the following Lemma is true.

\begin{lemma} \label{H2}
	There is a constant $c$ such that
	$$\lim_{|J| \to \infty} P\Big( M(J) \geq c |J| \Big) = 1.$$
\end{lemma}

\begin{proof}
	For notational simplicity in this proof we assume that $J=[0,N]$ and denote $M(N) = M([0,N])$. We have
	$$M(N) = \sum_{x\in \mathcal{P}} \mathbf{1}_{\{x \in [0,N]\}}  \mathbf{1}_{\{x \nleftrightarrow y \, \, \forall y \in  \mathcal{P} \} }. $$
	As before, by the Mecke formula \cite[Theorem 4.4]{LP}
	$$ \mathbb{E} [M(N)] = c \int_0^N P(x \nleftrightarrow y \, \, \forall y \in  \mathcal{P}) dx$$
	with the probability $P(x \nleftrightarrow y \, \, \forall y \in  \mathcal{P})$ given by
	\begin{align*} 
	& \sum_{n \in \N_0} 	P\Big( x \nleftrightarrow y \, \, \forall y \in  \mathcal{P} \Big| [0,N] \textnormal{ contains } n \textnormal{ Poisson points} \Big) \\
	& \quad \quad\quad\times P \Big( [0,N] \textnormal{ contains } n \textnormal{ Poisson points} \Big) \\
	&= \sum_{n \in \N_0} \prod_{k=1}^{N}	P\Big( x \nleftrightarrow y \, \, \forall y \in  \mathcal{P} \cap [k-1,k) \Big| [0,N] \textnormal{ contains } n \textnormal{ Poisson points} \Big) \\
	& \quad \quad\quad\times P \Big( [0,N] \textnormal{ contains } n \textnormal{ Poisson points} \Big) \\
	& \geq\sum_{n \in \N_0} \prod_{k=1}^{N}\exp \Big( -\beta n \frac{1}{(k-1)^2} \Big) P \Big( [0,N] \textnormal{ contains } n \textnormal{ Poisson points} \Big) \\
	& = \sum_{n \in \N_0} \exp \Big( -\beta n \sum_{k=1}^{N} \frac{1}{(k-1)^2} \Big) P \Big( [0,N] \textnormal{ contains } n \textnormal{ Poisson points} \Big) \\
	& \geq \sum_{n \in \N_0} \exp (-c\beta n) P \Big( [0,N] \textnormal{ contains } n \textnormal{ Poisson points} \Big) \equiv c >0,
	\end{align*}
	yielding that $ \mathbb{E} [M(N)] \geq cN$. Now we estimate the second moment of $M(N)$. Assume that $J_k=[k-1,k)$ for $k=1, \ldots, N$. We write
	\begin{align} \label{H2p1}
	\begin{split}
	M(N)^2 & = \sum_{k,j=1}^N \sum_{x\in \mathcal{P}} \sum_{v\in \mathcal{P}} \mathbf{1}_{\{x \in J_j\}} \mathbf{1}_{\{v \in J_k\}}  \mathbf{1}_{\{x \nleftrightarrow y \, \, \forall y \in  \mathcal{P} \} }  \mathbf{1}_{\{v \nleftrightarrow y \, \, \forall y \in  \mathcal{P} \} } \\
	& = \sum_{\substack{k,j=1\\ d(J_k, J_j) \leq \sqrt{N}}}^N \sum_{x\in \mathcal{P}} \sum_{v\in \mathcal{P}} \mathbf{1}_{\{x \in J_j\}} \mathbf{1}_{\{v \in J_k\}}  \mathbf{1}_{\{x \nleftrightarrow y \, \, \forall y \in  \mathcal{P} \} }  \mathbf{1}_{\{v \nleftrightarrow y \, \, \forall y \in  \mathcal{P} \} } \\
	& \quad + \sum_{\substack{k,j=1\\ d(J_k, J_j) > \sqrt{N}}}^N \sum_{x\in \mathcal{P}} \sum_{v\in \mathcal{P}} \mathbf{1}_{\{x \in J_j\}} \mathbf{1}_{\{v \in J_k\}}  \mathbf{1}_{\{x \nleftrightarrow y \, \, \forall y \in  \mathcal{P} \} }  \mathbf{1}_{\{v \nleftrightarrow y \, \, \forall y \in  \mathcal{P} \} } ,
	\end{split}
	\end{align}
	where $d(J_k, J_j)$ denotes the Euclidean distance between $J_k$ and $J_j$. Observe that
	\begin{align}\label{H2p2}
	\sum_{\substack{k,j=1\\ d(J_k, J_j) \leq \sqrt{N}}}^N\mathbb{E} \Big[ \sum_{x\in \mathcal{P}} \sum_{v\in \mathcal{P}} \mathbf{1}_{\{x \in J_j\}} \mathbf{1}_{\{v \in J_k\}}  \mathbf{1}_{\{x \nleftrightarrow y \, \, \forall y \in  \mathcal{P} \} }  \mathbf{1}_{\{v \nleftrightarrow y \, \, \forall y \in  \mathcal{P} \} } \Big]	\leq c N^{\frac{3}{2}},
	\end{align}
	since there are at most $c N^{\frac{3}{2}}$ pairs of intervals $J_k$ and $J_j$ with $d(J_k, J_j) \leq \sqrt{N}$. Now we estimate the expected value
	\begin{align*}
	& \mathbb{E} \Big[ \sum_{\substack{k,j=1\\ d(J_k, J_j) > \sqrt{N}}}^N \sum_{x\in \mathcal{P}} \sum_{v\in \mathcal{P}} \mathbf{1}_{\{x \in J_j\}} \mathbf{1}_{\{v \in J_k\}}  \mathbf{1}_{\{x \nleftrightarrow y \, \, \forall y \in  \mathcal{P} \} }  \mathbf{1}_{\{v \nleftrightarrow y \, \, \forall y \in  \mathcal{P} \} } \Big] \\
	& = \sum_{\substack{k,j=1\\ d(J_k, J_j) > \sqrt{N}}}^N \mathbb{E} \Big[  \sum_{x\in \mathcal{P}} \sum_{v\in \mathcal{P}} \mathbf{1}_{\{x \in J_j\}} \mathbf{1}_{\{v \in J_k\}}  \mathbf{1}_{\{x \nleftrightarrow y \, \, \forall y \in  \mathcal{P} \} }  \mathbf{1}_{\{v \nleftrightarrow y \, \, \forall y \in  \mathcal{P} \} } \Big] .
	\end{align*}
	For simplicity let $J_k =[0,1]$ and $J_j = [1+ \sqrt{N}, 2+\sqrt{N}]$. As before, using the Mecke formula
	\begin{align*}
	& \mathbb{E} \Big[  \sum_{x\in \mathcal{P}} \sum_{v\in \mathcal{P}} \mathbf{1}_{\{x \in J_j\}} \mathbf{1}_{\{v \in J_k\}}  \mathbf{1}_{\{x \nleftrightarrow y \, \, \forall y \in  \mathcal{P} \} }  \mathbf{1}_{\{v \nleftrightarrow y \, \, \forall y \in  \mathcal{P} \} } \Big] \\
	& = c \int_0^1 \int_{1+ \sqrt{N}}^{2+ \sqrt{N}} P \big( x \nleftrightarrow y \, \, \forall y \in  \mathcal{P}, \, \, v \nleftrightarrow y \, \, \forall y \in  \mathcal{P} \big) dx dv \\
	& =  c \int_0^1 \int_{1+ \sqrt{N}}^{2+ \sqrt{N}} P \big( x \nleftrightarrow y \, \, \forall y \in  \mathcal{P} \big| v \nleftrightarrow y \, \, \forall y \in  \mathcal{P} \big) P\big( v \nleftrightarrow y \, \, \forall y \in  \mathcal{P} \big) dx dv \\
	& =  c \int_0^1 \int_{1+ \sqrt{N}}^{2+ \sqrt{N}} P \big( x \nleftrightarrow y \, \, \forall y \in  \mathcal{P} \setminus \{v\} \big) P\big( v \nleftrightarrow y \, \, \forall y \in  \mathcal{P} \big) dx dv \\
	& =  c \int_0^1 \int_{1+ \sqrt{N}}^{2+ \sqrt{N}} \frac{P ( x \nleftrightarrow y \, \, \forall y \in  \mathcal{P} )}{P(x \nleftrightarrow v)} P\big( v \nleftrightarrow y \, \, \forall y \in  \mathcal{P} \big) dx dv \\
	& \leq c \int_0^1 \int_{1+ \sqrt{N}}^{2+ \sqrt{N}} {P ( x \nleftrightarrow y \, \, \forall y \in  \mathcal{P} )} \exp \big( \beta \frac{1}{d(I_k, I_j)^2} \big) P\big( v \nleftrightarrow y \, \, \forall y \in  \mathcal{P} \big) dx dv \\
	& \leq c (1+\frac{1}{N}) \int_0^1 \int_{1+ \sqrt{N}}^{2+ \sqrt{N}} {P ( x \nleftrightarrow y \, \, \forall y \in  \mathcal{P} )} P\big( v \nleftrightarrow y \, \, \forall y \in  \mathcal{P} \big) dx dv   \\
	& = c (1+\frac{1}{N}) \mathbb{E}[M(N)]^2,
	\end{align*}
	where we used the Mecke formula again in the last equality. Combining this with \eqref{H2p1} and \eqref{H2p2} we get
	$$ \mathbb{E} [M(N)^2] \leq c N^{\frac{3}{2}} + c (1+\frac{1}{N}) \mathbb{E}[M(N)]^2,$$
	which yields that $\operatorname{Var} (M(N)) \leq cN^{\frac{3}{2}}.$ An application of Chebyshev's inequality then easily shows that $M(N) \geq cN$ with high probability.
\end{proof}
Again let $J \subset [0,N]$ be an interval and assume that $J_k=[k-1,k)$ for $k=1, \ldots, N$. Now we say that a vertex $x \in \mathcal{P} \cap J \cap J_k$, for some $k \in \{1, \ldots N\}$, is almost isolated with respect to $J$ if either there is no edge connecting $x$ to another point $y \in \mathcal{P} \cap J$ or $x$ is only connected to points in $J \cap (J_{k-1} \cup J_{k+1})$ if $2 \leq k \leq N-1$, to points in $J \cap J_2$ if $k=1$ or to points in $J \cap J_{N-1}$ if $k=N$. We denote by $K(J)$ the number of such vertices.
Lemma \ref{H2} implies the following result.
\begin{cor} \label{H23}
	There is a constant $c$ such that
	$$\lim_{|J| \to \infty} P\Big( K(J) \geq c |J| \Big) = 1.$$
\end{cor}

Having established Lemma \ref{H2} and Corollary \ref{H23} our strategy now is to show that the bound $D(N) \geq N^\eta$ holds with high probability for a certain constant $0 < \eta < 1$. We say that an interval $J_j$ of the form $J_j = [a,b]$ is a cut interval if the number of edges from $[0,a]$ to $[b,N]$ is zero. From a calculation similar to the one made in the proof of Lemma \ref{H2} we get that the expected number of cut intervals is at least $cN$. We now show that $D(N) \geq N^\eta$ with high probability for a certain constant $0 < \eta < 1$. Indeed, divide $[0,N]$ into $N^{\frac{2}{3}}$ intervals $J_1, \ldots, J_{N^{\frac{2}{3}}}$ each of side length $N^{\frac{1}{3}}$. By Corollary \ref{H23} the number of almost isolated intervals is at least $cN^{\frac{2}{3}}$ with high probability. Fix $J_i = [a,b]$. We say that a vertex point $x \in J_i$ is a local cut point if it is almost isolated with respect to $J_i$. Let $C_i = C(J_i)$ be the number of local cut points. By Corollary \ref{H23} we have that $\mathbb{E}[C_i] \geq N^{\frac{1}{3}}$. Moreover, if $P_{J_i}$ denotes the number of Poisson points in $J_i$ then
\begin{align*}
\operatorname{Var} (C_i) & \leq \mathbb{E}[C_i^2] \leq \mathbb{E}[P_{J_i}^2] \\
& = \rho^2 |J_i|^2 + \rho |J_i| \leq c |J_i|^2.
\end{align*}
Thus, as argued before, an application of Chebyshev's inequality yields that there is an almost isolated interval which contains at least $cN^{\frac{1}{3}}$ local cut points with high probability. Denote this interval by $J_{i^*}$. Let $K=K(J_{i^*})$ be the number of edges between $J_{i^*}$ and $J_{i^*-1}$. We estimate $\mathbb{E}[ K(J_{i^*})]$ as follows. We obtain
\begin{align*}
\mathbb{E}[ K(J_{i^*})] & \leq c \int_0^{N^{\frac{1}{3}}} \int_{N^{\frac{1}{3}}}^{2N^{\frac{1}{3}}} g(x-y) dx dy \\
& \leq c \int_0^{N^{\frac{1}{3}}} \int_{N^{\frac{1}{3}} + \frac{1}{N^{\frac{1}{3}}}}^{2N^{\frac{1}{3}}} (x-y)^{-2} dx dy + c \int_0^{N^{\frac{1}{3}}} \int_{N^{\frac{1}{3}}}^{N^{\frac{1}{3}} + \frac{1}{N^{\frac{1}{3}}}} dx dy \\
& \leq c \int_0^{N^{\frac{1}{3}}} (N^{\frac{1}{3}} + \frac{1}{N^{\frac{1}{3}}}-y)^{-1} dx dy + c \\
& \leq c \log ( N^{\frac{1}{3}} + \frac{1}{N^{\frac{1}{3}}}) + \log (N^{\frac{1}{3}}) \leq c \log N.
\end{align*}
A similar calculation shows that the same estimate holds for the expected number of edges from $J_{i^*}$ to $J_{i^*+1}$. By an application of the Markov inequality we conclude that there exist at most $2 \log^2 N$ edges exiting the almost isolated interval $J_{i^*}$ with high probability. Recall that there are at least $c N^{\frac{1}{3}}$ local cut points with respect to $J_{i^*}$ with high probability. Thus there exist two local cut points, say $i_1$ and $i_2$ such that $[i_1, i_2]$ contains more than $c N^{\frac{1}{3}} / \log^2 N \geq c N^{\eta}$ local cut points for some $\eta \in (0,1)$ and there are no edges exiting $[i_1, i_2]$. Take the $(\frac{1}{3})L^{th}$ and $(\frac{2}{3})L^{th}$ local cut points in $[i_1, i_2]$, with $L$ denoting the number of local cut points in $[i_1, i_2]$. By definition the length of the shortest path between such points is at least $\frac{1}{3}L \geq c N^{\eta}$ with high probability. This implies that $D(N) \geq cN^{\eta}$ with high probability. We finished the proof of the lower bound. \hfill $\Box$

\section{The case $d < s < 2d$}
As explained in the introduction it suffices to prove the lower bound. This will follow from the following assertion, which is the analogue of \cite[Proposition 7.4]{DW}.

\begin{lemma} \label{H3}
	Denote by $P_{0,x}$ the probability measure given by $P_{0,x}( \cdot) = P( \cdot | 0,x \in \mathcal{P})$. There exists a constant $c \in (0, \infty)$ such that
	$$ \lim_{\|x\| \to \infty} P_{0,x} \Big(  D(0,x) \geq c\log \|x\|  \Big) = 1.$$
\end{lemma}

\begin{proof}
	Let $n \in \mathbb{N}$ and $0,x \in \mathcal{P}$. Define $x_0 =0$ and $x_n =x$. Denote by $\mathbb{E}_{0,x}$ the expected value with respect to $P_{0,x}$. We have
	\begin{align*}
	P_{0,x} \big( D(0,x) = n \big) & \leq \mathbb{E}_{0,x} \big[ \mathbf{1}_{\{\text{there is a path from $x_0$ to $x_n$ of length $n$}\}} \big] \\
	& \leq \mathbb{E}_{0,x} \Big[ \sum_{x_1, \ldots, x_{n-1} \in \mathcal{P}} \mathbf{1}_{\{x_0 \leftrightarrow x_1 \leftrightarrow \ldots \leftrightarrow x_{n-1} \leftrightarrow x_n\}} \Big] \\
	& \leq c \rho^{n-1} \int_\rd dx_1 \ldots \int_\rd dx_{n-1} \prod_{i=1}^n g(x_i - x_{i-1}),
	\end{align*}
	where we used the Mecke formula in the last step. Using the substitution of $x_i$ by $x_i - \sum_{k=1}^{i-1} x_k$ inductively for $i=1, \ldots, n-1$ we further obtain
	\begin{align*}
	P_{0,x} \big( D(0,x) = n \big) & \leq c \rho^{n-1} \int_\rd dx_1 \ldots \int_\rd dx_{n-1} \prod_{i=1}^n h(x_i - x_{i-1}) \\
	& \leq  c \rho^{n-1} \int_\rd dx_1 \ldots \int_\rd dx_{n-1} \prod_{i=1}^{n-1} h(x_i) h\Big( x_i - \sum_{k=1}^{n-1} x_{k} \Big)
	\end{align*}
	for the function $h(x) = \min (1, \| x\|^{-s})$, $x \in \rd$. Note that the fact $\|x_i\| < \frac{\|x\|}{n}$ for all $i=1, \ldots, n-1$ implies $\| x - \sum_{k=1}^{n-1} x_{k}\| \geq \frac{\|x\|}{n}$ and $h( x_i - \sum_{k=1}^{n-1} x_{k} )$ can be bounded from above by $\sup_{y : \|y\| \geq \frac{\|x\|}{n}} h(y)$. Otherwise, if there is at least one $i \in \{1, \ldots, n-1\}$ such that $\|x_i\| \geq \frac{\|x\|}{n}$ in such a case we can bound $h(x_i)$ by $\sup_{y : \|y\| \geq \frac{\|x\|}{n}} h(y)$. Using this observation we derive
	\begin{align} \label{H3p1}
	P_{0,x} \big( D(0,x) = n \big) & \leq c n \sup_{y : \|y\| \geq \frac{\|x\|}{n}} h(y) \Big( \rho \int_\rd h(y) dy \Big)^{n-1}.
	\end{align}
	Note that
	$$\int_\rd h(y) dy  < \infty,$$
	since $s>d$ by assumption. For arbitrary, but fixed $\mu \in (0, \infty)$ choose $\|x\|$ large enough such that $\mu \log \|x\| \geq n$. Then
	$$ \sup_{y : \|y\| \geq \frac{\|x\|}{n}} h(y) \leq \mu^s (\log \|x\|)^s \| x\|^{-s}.$$
	Inserting this inequality into \eqref{H3p1} we obtain for $\|x\|$ sufficiently large
	\begin{align*}
	P_{0,x} \big( D(0,x) = n \big) & \leq c n \Big( \rho \int_\rd h(y) dy \Big)^{n-1} \mu^s (\log \|x\|)^s \| x\|^{-s} \\
	& \leq \Big( 1+ \rho \int_\rd h(y) dy \Big)^{\mu \log \|x\|} \mu^{s+1} (\log \|x\|)^{s+1} \|x \|^{-s} \\
	& = \mu^{s+1} (\log \|x\|)^{s+1} \|x\|^{-s + \mu \log (1+ \rho \int_\rd h(y) dy)} \leq \|x \|^{-\delta}
	\end{align*}
	for some constant $\delta \in (0, \infty)$, once $\mu$ is sufficiently small satisfying $$\mu \log \Big(1+ \rho \int_\rd h(y) dy\Big) < s.$$
	We conclude that
	\begin{align*}
	P_{0,x} \big( D(0,x) \leq c \log \|x\| \big) = \sum_{n=1}^{c \log \|x\|} 	P_{0,x} \big( D(0,x) = n \big) \leq c \big( \log \|x\| \big) \|x \|^{-\delta}
	\end{align*}
	for some sufficiently small constant $c$. From this the assertion follows.
\end{proof}

\section{The case $s=d$}

We begin by proving the lower bound.

\begin{remark} \label{mr1}
	As mentioned in the discussions after Theorem \ref{MR}, due to the proof of \cite[Theorem 3.4]{DW} it follows that its statement is also true in the present model under investigation. The latter Theorem states that for the component $\mathcal{C}_N$, under the assumption $s\in(d,2d)$, we have $|\mathcal{C}_N| \geq cN^d$ for some $c\in (0, \infty)$ with high probability. As before $|\mathcal{C}_N|$ denotes the size, i.e. number of nodes, of the cluster $|\mathcal{C}_N|$. Observe that, since the probability of connecting two vertices is monotonically decreasing in $s$, the status of the edges, and in particular the size of the component $\mathcal{C}_N$ in case $s=d$, stochastically dominates the one in case $s\in(d,2d)$. Thus, in the following we can make use of the fact that $|\mathcal{C}_N| \geq cN^d$ for some $c\in (0, \infty)$ with high probability.
\end{remark}

\subsection{Lower bound}

Assume that $I_1, \ldots, I_{N^d}$ denote subcubes  each of unit side length and are a partition of $[0,N]^d$. Let $l \in \{1, \ldots, N^d\}$ and take one of these subcubes $I_l$ and for $m \in \N$ define $B(m)$ as the total number of nodes in $I:=\bigcup_{k=1,\\ \substack{k\neq l}}^{N^d} I_k$ reachable from $I_l$ by paths with length at most $m$. Without loss of generality assume that $l=1$. Then
\begin{align*}
B(m) \leq \sum_{j=1}^m \sum_{x\in \mathcal{P}} \mathbf{1}_{\{x \in I_1\}} \sum_{y_1\in \mathcal{P}}  \mathbf{1}_{\{y_1 \in I\}} \ldots \sum_{y_j\in \mathcal{P}}  \mathbf{1}_{\{y_j \in I\}} \mathbf{1}_{\{x \leftrightarrow 
	y_1 \leftrightarrow y_2 \leftrightarrow \ldots \leftrightarrow y_{j-1} \leftrightarrow y_j\}} .
\end{align*}
As before, from the Mecke formula we get
\begin{align} \label{6p1}
\begin{split}
\mathbb{E} [B(m)] & \leq c \sum_{j=1}^m \int_{I_1} dx \int_I dy_1 \ldots \int_I dy_j P\big( x \leftrightarrow 
y_1 \leftrightarrow y_2 \leftrightarrow \ldots \leftrightarrow y_{j-1} \leftrightarrow y_j \big) \\
& \leq c \sum_{j=1}^m \int_{I_1} dx \int_I dy_1 \ldots \int_I dy_j \min (1, \| x-y_1\|^{-d}) \prod_{l=2}^j \min (1, \| y_l-y_{l-1}\|^{-d}) .
\end{split}
\end{align}
By a change to polar coordinates we get
\begin{align*}
\int_I \min (1, \| y_l-y_{l-1}\|^{-d}) dy_l &\leq c\log N, \quad l=2, \ldots, j, \\
\text{ and }  \int_I \min (1, \| x-y_{1}\|^{-d}) dy_1 &\leq c\log N.
\end{align*}
Hence, inserting this into \eqref{6p1}
\begin{align*}
\mathbb{E} [B(m)] \leq c \sum_{j=1}^m (\log N)^j \leq c \frac{(\log N)^{m+1} - 1}{(\log N) - 1} \leq c (\log N)^m.
\end{align*}
From the Markov inequality we deduce that
$$ P (B(m) \geq cN^d) \leq \frac{\mathbb{E}_{0,x} [B(m)]}{cN^d} \to 0$$
as $N \to \infty$ for
$$ m = \frac{(1-\varepsilon) \log N}{\log\log N},$$
where $\varepsilon >0$ is arbitrarily small and $c$ denotes a positive constant. From this we derive that there exists at least one node, which is connected to $I_1$ only by paths of length at least $m$ with high probability. Hence, together with Remark \ref{mr1} we conclude that $D(N) \geq cm$ with high probability as desired.

\subsection{Upper bound}
We now give a proof of the upper bound based on an idea presented in \cite{CGS}. Roughly spoken, the idea is that all paths of length $m$ for suitably chosen $m$ must have an endpoint inside a cube which is not too far away from the origin. Applying a result we established in the preceding section the distance within small cubes of the infinite cluster is also small enough. Combining these findings we will arrive at the desired upper bound.

We now make the discussion more precise. At first we investigate arbitrary paths of length $m$ with $m$ to be chosen later in this proof. Fix one element $x \in \mathcal{C}_\infty$ with $\|x\| \leq N$ for sufficiently large $N$. Consider all the paths $y_1 \leftrightarrow y_2 \leftrightarrow x$ of length 2 ending in $x$. Let 
$$ X_1 := \arg \min_{y_1 : y_1 \leftrightarrow y_2 \leftrightarrow x} \|y_1\|,$$
i.e. $X_1$ is the smallest, in norm, vertex connected to $x$ via a path of length at most 2. By construction we have $\|X_1\| \leq \|x\|$. Similarly we obtain a (random) vertex which we call $X_2$ connected to $X_1$ by a path of length $2$ with the property that $\|X_2\| \leq \|X_1\|$. We continue the procedure in this fashion for a total number of $m$ times. Our first target in this section is the following statement, which is an analogue to \cite[Lemma 10.1]{CGS}.

\begin{lemma} \label{H4}
	If $m = (2d+2) 2^{c+1} \frac{\log N}{\log \log N}$ for a sufficiently large constant $c$ then
	$$\| X_m\| \leq \exp \big( (\log N)^\frac{d}{2^c} \big)$$
	with probability at least $1 - \frac{1}{N^{2d}}.$
\end{lemma}
\begin{proof}
	Following the calculations made in \cite[p. 337]{CGS} Lemma \ref{H4} follows from Lemma \ref{H5} below. We omit the details.
\end{proof}

\begin{lemma} \label{H5}
	Let $\|y\| > \exp ( (\log N)^\frac{d}{2^c} )$.Then we have
	$$ \mathbb{E} \Big[\| X_r\| \Big| X_{r-1} = y \Big] \leq c \frac{\|y\|}{(\log N)^\frac{1}{2^{c+1}}}.$$
\end{lemma}
\begin{proof}
	Denote by $B(y)$ the total number of nodes which are connected to $X_{r-1}=y$ and which have a norm smaller than $\|y\|$, i.e.
	\begin{align*}
	B(y) = \sum_{v\in \mathcal{P}} \mathbf{1}_{\{ \|v\| \leq \|y\| \}} \mathbf{1}_{\{v \leftrightarrow y\}} .
	\end{align*}
	Note that $\|v-y\| \leq 2 \|y\|$ for each such node $v$ by the triangle inequality. We now show that there is a constant $c \in (0, \infty)$ such that
	$$c \log \|y\| \leq B(y)$$
	with high probability. We first calculate the expected value of $B(y)$. We get for sufficiently large $\|y\|$ that
	\begin{align*}
	\mathbb{E} [B(y)] &= c \int_{0 \leq \|v\| \leq \|y\|} g(v-y)dv = c \Big( \int_0^1 g(z) dz + \int_1^{\|y\|} \|z\|^{-d} dz \Big) \\
	& = c + c \log \|y\|. 
	\end{align*}
	Next we calculate the expected value of
	\begin{align*}
	B(y)^2 = \sum_{v\in \mathcal{P}} \sum_{w\in \mathcal{P}} \mathbf{1}_{\{ \|v\| \leq \|y\| \}} \mathbf{1}_{\{ \|w\| \leq \|y\| \}} \mathbf{1}_{\{v \leftrightarrow y\}} \mathbf{1}_{\{w \leftrightarrow y\}} .
	\end{align*}
	Here we get
	\begin{align*}
	\mathbb{E} [B(y)^2] & = \mathbb{E} \Big[ \sum_{v\in \mathcal{P}} \mathbf{1}_{\{ \|v\| \leq \|y\| \}}\mathbf{1}_{\{v \leftrightarrow y\}} \Big]  + \mathbb{E} \Big[\sum_{\substack{v,w \in \mathcal{P}\\ v \neq w}} \mathbf{1}_{\{ \|v\| \leq \|y\| \}} \mathbf{1}_{\{ \|w\| \leq \|y\| \}} \mathbf{1}_{\{v \leftrightarrow y\}} \mathbf{1}_{\{w \leftrightarrow y\}} \Big] \\
	& = \mathbb{E}[B(y)]  + \mathbb{E} \Big[\sum_{\substack{v,w \in \mathcal{P}\\ v \neq w}} \mathbf{1}_{\{ \|v\| \leq \|y\| \}} \mathbf{1}_{\{ \|w\| \leq \|y\| \}} \mathbf{1}_{\{v \leftrightarrow y\}} \mathbf{1}_{\{w \leftrightarrow y\}} \Big] \\
	& \leq \mathbb{E}[B(y)] + c \int_0^{\|y\|} \int_0^{\|y\|}  g(v-y) g(w-y)dvdw \\
	& = \mathbb{E}[B(y)] + \big(\mathbb{E}[B(y)]\big)^2,
	\end{align*}
	using which yields that $\operatorname{Var} (B(y)) \leq \mathbb{E} [B(y) ].$
	An application of Chebyshev's inequality then shows that
	\begin{align} \label{P1}
	P \Big( B(y) \leq c \log \|y\| \Big) \leq c \frac{1}{(\log N)^\frac{d}{2^c}}
	\end{align}
	for $\|y\| > \exp ( (\log N)^\frac{d}{2^c} ).$ We proceed as follows. Define
	$$ V(y) := \Big\{ z \in \rd : \| z \| \leq \frac{\|y\|}{(\log N)^\frac{1}{2^{c+1}}} \Big\}.$$
	Then the volume of $V(y)$ is at least $$| V(y)| = c  \frac{\|y\|^d}{(\log N)^\frac{d}{2^{c+1}}}.$$ Let $v$, $\|v\| \leq \|y\|$, be a vertex connected to $y$ (if it exists). Then $\|v-y\| \leq 2\|y\|$. We estimate the probability that $y$ is connected to no vertex in $V(y)$
	\begin{align*}
	P\big( v \nleftrightarrow z \, \, \forall z \in V(y) \big) & = \sum_{n \in \N_0} P\Big( v \nleftrightarrow z \, \, \forall z \in V(y) \Big| V(y) \textnormal{ contains } n \textnormal{ Poisson points} \Big) \\
	& \quad \quad P\big(V(y) \textnormal{ contains } n \textnormal{ Poisson points}\big) \\
	& \leq \sum_{n \in \N_0} \exp \big( -\beta n \frac{1}{2\|y\|^d} \big) P\big(V(y) \textnormal{ contains } n \textnormal{ Poisson points}\big) \\
	& = \exp \Big( -c | V(y)| \big( 1-e^{-\beta/(2\|y\|^d)} \big) \Big)
	\end{align*}
	Since $$\lim_{x \to 0} \frac{x}{1-e^{-x}} = 1,$$ for $0 \leq x \leq 1$ we can choose a constant $c$ such that $x \leq  c(1-e^{-x})$. We may choose $\|y\|$ sufficiently large such that $ \frac{\beta}{2\|y\|^d} \leq 1$. Thus, applying this inequality we arrive at
	\begin{align*}
	P\big( v \nleftrightarrow z \, \, \forall z \in V(y) \big) \leq \exp \Big( -c | V(y)| \frac{\beta}{2\|y\|^d} \Big) = \exp \Big( -\frac{c}{(\log N)^\frac{d}{2^{c+1}}} \Big).
	\end{align*}
	By \eqref{P1} with high probability $y$ has at least $c \log \|y\|$ vertices $v$ with $\|v\| \leq \|y\|$ connected to it. Conditioned on this event, the probability that no vertex in $V(y)$ is connected to $y$ by a path of length two is at most $$\exp\Big(-c \frac{\log\|y\|}{(\log N)^\frac{d}{2^{c+1}}}\Big) \leq \exp\Big( -(\log N)^\frac{d}{2^{c+1}} \Big)$$
	for $\|y\| > \exp((\log N)^\frac{d}{2^{c}})$. Hence, the probability that no node in $V(y)$ is connected to $y$ by a path of length two, is at most
	$$ c \frac{1}{(\log N)^\frac{d}{2^{c}}} + \exp\Big( -(\log N)^\frac{d}{2^{c+1}} \Big) \leq c \frac{1}{(\log N)^\frac{d}{2^{c}}} .$$
	In summary, conditioned on $X_{r-1} =y$ the bound $\|X_r\| \leq \frac{\|y\|}{(\log N)^\frac{1}{2^{c+1}}}$ holds with high probability. On the other hand, with probability one $\|X_r\| \leq \|X_{r-1}\|.$ This concludes the proof.
\end{proof}

Now we are in position to complete the proof of the upper bound. \newline \newline
\textit{Proof of Theorem \ref{MR}, part 4 (ii).} Again we remark that, since the probability of connecting two vertices is monotonically decreasing in $s$, the status of the edges in case $s=d$, stochastically dominates the one in case $s\in(d,2d)$. In particular, this implies that the graph distance between two vertices in case $s\in(d,2d)$ stochastically dominates the one in case $s=d$, i.e. the graph distances become smaller for $s=d$. Hence, we observe that in applying part 3 of the main Theorem we have for $x \in \mathcal{C}_\infty$ with $\|x\| = \exp ( (\log N)^\frac{d}{2^c} )$ and some $\delta \in (1, \infty)$ that
$$ D(0,x) \leq \big( \log N)^\frac{d}{2^c} \big) ^\delta \leq \log^{\frac{1}{2}} N$$
with high probability as soon as $\frac{2^c}{d} \geq 2\delta$ for some constant $c$. In particular
$$D(0,x) \leq c \frac{\log N}{\log \log N} , \quad \|x\| = \exp ( (\log N)^\frac{d}{2^c} ),$$
with high probability. But by Lemma \ref{H4} with probability at least $ 1- \frac{1}{N^{2d}}$ a vertex $y \in \mathcal{C}_\infty$ with $\|y\| =N$ is connected to such a vertex $x$ by a path of length $m \leq c\frac{\log N}{\log \log N}$. Thus, for $\|y\| =N$ and $\|x\| =\exp ( (\log N)^\frac{d}{2^c} )$ we have 
$$ D(x,y) \leq c \frac{\log N}{\log \log N}$$
with high probability. The fact that
$$ D(0,y) \leq D(0,x) +D(x,y)$$
implies the desired result. \hfill $\Box$

\section*{Acknowledgements}
The author would like to thank the anonymous referees and the Editor for their constructive comments that improved the quality of this paper.

\bibliographystyle{plain}

\end{document}